\documentclass[11pt,reqno]{amsart} 

\usepackage{hyperref}

\usepackage[margin=1in]{geometry}
\usepackage{tikz}    
\usepackage{graphicx, epsfig, textcomp, txfonts, pxfonts}
\usepackage{amsmath,amsthm,amsfonts,amssymb,amsxtra, mathtools, dsfont} 
\usepackage{mathrsfs}
\usepackage{mathdots}
\usepackage{tikz-cd}
\usepackage{dsfont}

\usepackage{epstopdf}
\usepackage[caption=false]{subfig}

\usepackage[square, numbers]{natbib}

\newtheorem{theorem}{Theorem}[subsection]
\newtheorem*{theorem*}{Theorem}
\newtheorem{lemma}[theorem]{Lemma}
\newtheorem{corollary}[theorem]{Corollary}
\newtheorem{proposition}[theorem]{Proposition}

\newtheorem{definition}[theorem]{Definition}
\newtheorem{example}[theorem]{Example}
\newtheorem{remark}[theorem]{Remark}

\newcommand{\C}{\mathbb{C}}

\newcommand{\Z}{\mathbb{Z}}

\newcommand{\gl}{\mathfrak{gl}}
\renewcommand{\k}{\mathds{k}}
\newcommand{\0}{\bar{0}}
\newcommand{\1}{\bar{1}}
\newcommand{\BasisB}{\mathcal{B}}
\newcommand{\parity}[1]{\overline{#1}}
\newcommand{\Id}{\operatorname{Id}}
\newcommand{\wreath}{\wr}
\newcommand{\bv}{\mathbf{v}}

\newcommand{\SdwreathA}{S_{d} \wreath A}
\newcommand{\Hom}{\operatorname{Hom}}
\newcommand{\bi}{\mathbf{i}}
\newcommand{\bj}{\mathbf{j}}

\newcommand{\ba}{\mathbf{a}}
\newcommand{\bb}{\mathbf{b}}
\newcommand{\bc}{\mathbf{c}}
\newcommand{\bepsilon}{\boldsymbol{\varepsilon}}
\newcommand{\bone}{\mathbf{1}}
\newcommand{\End}{\operatorname{End}}
\newcommand{\Itwo}{\tilde{I}^{2}(n,d)}
\newcommand{\Cliff}{\operatorname{Cliff}}
\newcommand{\Pol}{\mathcal{P}}

\newcommand{\fq}{\mathfrak{q}}
\newcommand{\fh}{\mathfrak{h}}

\newcommand{\tI}{\tilde{I}}
\numberwithin{equation}{subsection}

\title[Schur--Weyl Equivalences for Wreath Product Superalgebras]{Schur--Weyl Equivalences for Wreath Product Superalgebras}

\author{Lauren Grimley}
\address{Department of Mathematical Sciences,
		University of Oklahoma,
		Norman, OK 73069, USA}
\email{lgrimley@ou.edu}

\author{Jonathan R.\  Kujawa}
\address{Department of Mathematics \\
		Oregon State University \\
		Corvallis, OR 97331, USA}
\email{kujawaj@oregonstate.edu}

\thanks{JRK was supported in part by Simons Foundation Grant No.\ 525043 and No.\ SFI-MPS-TSM-00014272.}

\date{\today}


\allowdisplaybreaks

\begin{document}

\begin{abstract} Let $A$ be an associative superalgebra over a field of characteristic zero.  Let $n \geq d+1$.  The main result of the paper establishes an equivalence between the category of supermodules for the wreath product $\SdwreathA$ and an explicitly defined category of supermodules for the general linear Lie superalgebra $\gl_{n}(A)$.  We also give an example showing the bound $n \geq d+1$ cannot be improved.
\end{abstract}



\maketitle


\section{Introduction}

\subsection{Overview} Schur--Weyl duality has played a fundamental role in representation theory for more than century.  It provides an equivalence of categories between the complex representations of the symmetric group on $d$ letters, $S_{d}$, and the polynomial representations of degree $d$ for the general linear Lie algebra, $\gl_{n}(\C)$, whenever $n \geq d$.  Due to its importance, Schur--Weyl duality has been generalized to a wide variety of settings.  Relevant to this paper, these generalizations include when $\gl_{n}(\C)$ is replaced with affine and toroidal Lie algebras, or with the type $Q$ Lie superalgebra \cite{Flicker,CGLMUW,Sergeev}.  The symmetric group is then replaced with the affine symmetric group or Sergeev superalgebra, respectively.  The purpose of the present paper is to provide a uniform proof that includes all of these, along with many additional cases of interest.

\subsection{Main Result}  Let $\k$ be a characteristic zero field and let $A$ be a unital associative superalgebra over $\k$.  The prefix ``super'' indicates that $A$ is $\Z_{2}= \Z /2\Z$-graded.  Every algebra can be viewed as a superalgebra concentrated in parity $\0$, but including other superalgebras allows us to cover the finite and affine type Q Lie superalgebras, among others.

For $d \geq 1$, let $\SdwreathA$ denote the wreath product superalgebra of $\k S_{d}$ and $A$.  Wreath product superalgebras naturally generalize many of the nice features of the symmetric groups and have long been of interest.    On the Lie algebra side, for $n \geq 1$, let $\gl_{n}(A)$ denote the Lie superalgebra over $\k$ consisting of $n \times n$-matrices with entries from $A$ and Lie bracket given by the supercommutator.  This also includes many examples of interest; see Section~\ref{SS:Examples}.   We refer the reader to Sections~\ref{SS:LieSuperalgebras}~and~\ref{SS:WreathProductSuperalgebras} for details on the wreath product and Lie superalgebras.


Let $V_{n}= A^{\oplus n}$, viewed as column vectors of height $n$ with entries from $A$.  This is a $\gl_{n}(A)$-supermodule via left matrix multiplication.  The $d$-fold tensor product $V_{n}^{\otimes d}$ is then a $U(\gl_{n}(A))$-supermodule via the coproduct on the universal enveloping superalgebra $U(\gl_{n}(A))$.  It also has a natural right action by $\SdwreathA$.  These actions commute, making $V_{n}^{\otimes d}$ a $(U(\gl_{n}(A)), \SdwreathA )$-bisupermodule.  Consequently, there is a functor, 
\[
F^{A}_{n,d}: \SdwreathA \text{-supermodules} \to U(\gl_{n}(A))\text{-supermodules},
\] given by $F^{A}_{n,d}(N) = V_{n}^{\otimes d} \otimes_{\SdwreathA} N$.

In Definition~\ref{D:polynomial} we define a full subcategory, $\Pol_{d}(\gl_{n}(A))$, of $U(\gl_{n}(A))$-supermodules consisting of supermodules which are polynomial representations of degree $d$ when restricted to the subalgebra $\gl_{n}(\k) \subseteq \gl_{n}(A)$ and also satisfy a certain condition that examples show is necessary.  The image of $F^{A}_{n,d}$ is easily seen to land in $\Pol_{d}(\gl_{n}(A))$.  Our main theorem is the following result.

\begin{theorem*}\label{T:IntroMainTheorem} If $n \geq d+1$, then the functor 
\[
F^{A}_{n,d}: \SdwreathA \text{-supermodules} \to \Pol_{d}(\gl_{n}(A))
\] is an equivalence of categories.
\end{theorem*} 

When $A = \k [t_{1}^{\pm 1}, \dotsc , t_{r}^{\pm 1}]$ or $A$ is the Clifford superalgebra on one generator, then we essentially recover the aforementioned generalizations of classical Schur--Weyl duality.  Our result shows that we can further generalize to an arbitrary associative $\k$-superalgebra.  We also note that Example~\ref{Ex:Example2} shows that the bound $n \geq d+1$ is sharp.  Thus, in some sense, this is the most general result along these lines that one could hope for.

Our general approach is similar to that of \cite{Flicker,CGLMUW}: we prove the equivalence by bootstrapping from the classical Schur--Weyl equivalence.    In comparing our arguments to theirs we find an instance of George P\'{o}lya's Inventor's Paradox: it turns out that our more general situation makes for easier proofs. As far as we are aware, this bootstrapping technique goes back to Chari--Pressley \cite{ChariPressley}.

\subsection{Future Directions}\label{SS:FutureDirections}  Polynomial representations of $\gl_{n}(\k )$ play an important role in representation theory and Lie theory.  For well-chosen $A$, the category $\Pol_{d}(\gl_{n}(A))$ has desirable qualities (e.g., see Corollary~\ref{C:quasi-hereditary}) and would be interesting to study.

There is a Schur--Weyl equivalence between the Lie superalgebra $\gl_{m|n}(\C)$ and the symmetric group \cite{BereleRegev,Sergeev}.  We fully expect the methods used here should generalize to this case and yield an equivalence of categories between certain $\gl_{m|n}(A)$-supermodules and $\SdwreathA$-supermodules.  See \cite{Flicker2} where something like this is done for $A=\k[t^{\pm 1}]$.

Classical Schur--Weyl duality holds for any infinite field, and one could hope to prove a positive characteristic analogue of the main theorem.  Properly formulating a definition of the category $\Pol_{d}(\gl_{n}(A))$ will likely require working with an appropriate integral form for $U(\gl_{n}(A))$.  The existence of integral forms for the universal enveloping superalgebras of affine and map Lie (super)algebras suggests this should be possible \cite{BagciChamberlin,Chamberlin,Garland}.  The representation theory of quantized and/or affinized wreath product superalgebras and their associated quantized Schur algebras is likewise interesting.  See \cite{LNX,RossoSavage-Wreath,Savage,WanWang} and the references therein.  It would be interesting to obtain a quantized and/or affinized version of the main theorem.  See \cite{ChariPressley} for the quantized affine case.

\subsection{Acknowledgements}\label{SS:Acknowledgements} The authors thank Georgia Benkart for her myriad contributions to mathematics, many of which relate to this work.

\section{Preliminaries} 

\subsection{Notation and Conventions} For $t \geq 1$, let $[1,t] = \{1, \dotsc ,t \}$. For $n, d \geq 1$, let 
\[
I(n,d) = \left\{\bi = (i_{1}, \dotsc , i_{d}) \mid i_{1}, \dotsc , i_{d} \in [1,n] \right\}
\] denote the set of $d$-tuples with entries from $[1,n]$.
 Let $\tI(n,d) \subseteq I(n,d)$ be the subset of tuples that have pairwise distinct entries.      Let
\[
 \Itwo = \left\{(\bj, \bi)_{k}  \mid (\bj, \bi) \in \tI(n,d) \times \tI(n,d), k\in [1,d], \text{ $i_{t}=j_{t}$ for $t \neq k$} \right\}.
\]  Note that the entries $j_{k}$ and $i_{k}$ may or may not be equal for a pair $(\bj,\bi)_{k}$ in $\Itwo$.

Throughout, let $\k$ be a field of characteristic zero and let $A = A_{\0} \oplus A_{\1}$ be a unital associative $\k$-superalgebra.  We assume the reader is familiar with basic notions for superalgebras, supermodules, and other $\Z_{2}=\Z /2\Z$-graded structures, and we only make remarks as needed to establish our notation and our conventions.  Let $\BasisB$ denote a fixed homogeneous basis for $A$.  We always assume the identity of $A$ is an element of $\BasisB$.  We identify $\k$ as a subalgebra of $A$ via $c \mapsto c1=c1_{A}$.  For $d \geq 1$, set 
\begin{align*}
I(\BasisB,d) &= \left\{\bb = (b_{1}, \dotsc , b_{d}) \mid b_{1}, \dotsc , b_{d} \in \BasisB  \right\}, \\
I(A,d) &= \left\{\ba = (a_{1}, \dotsc , a_{d}) \mid a_{1}, \dotsc , a_{d} \in A_{\0} \cup A_{\1} \right\}.
\end{align*}  For short, write $\bone$ for $(1,1,\dotsc ,1) \in I(\BasisB,d)$.

Given a $\k$-superspace $W =  W_{\0}\oplus W_{\1}$  and a homogeneous element $w \in W$, we write $\parity{w} \in \Z_{2}$ for the parity of $w$.  For short, we may call $v$ \emph{even} (resp., \emph{odd}) if $\parity{w}=\0$ (resp., $\parity{w}=\1$). Given $\k$-superspaces $U$ and $W$, there is a natural $\Z_{2}$-grading on $U\otimes W$ given by declaring $(U\otimes W)_{t} = \oplus_{r+s = t} U_{r}\otimes W_{s}$ for $r,s,t \in \Z_{2}$.  Here and elsewhere, $-\otimes-=-\otimes_{\k}-$.

\subsection{Lie Superalgebras}\label{SS:LieSuperalgebras}
Let \(\gl_{n}(A)\) be the $\k$-superspace of $n \times n$ matrices with entries in $A$.  For $r,s \in [1,n]$ and $a \in A$, write  \(E_{r,s}^{a}\) for the $n \times n$ matrix with $(r,s)$-entry equal to $a$ and all other entries equal to zero.  The $\Z _{2}$-grading on $\gl_{n}(A)$ is given by declaring $\overline{E_{r,s}^{a}} = \overline{a}$ for all homogeneous $a \in A$.

The set of matrices $\{E_{r,s}^{b} \mid r,s \in [1,n] \text{ and } b \in \BasisB\}$ form a homogeneous $\k$-basis for $\gl_{n}(A)$.  We view $\gl_{n}(A)$ as a Lie superalgebra over $\k$ via the supercommutator:
\begin{equation}\label{E:supercommutator}
[E_{r_{1},s_{1}}^{a_{1}}, E_{r_{2},s_{2}}^{a_{2}}] = \delta_{s_{1}, r_{2}} E_{r_{1},s_{2}}^{a_{1}a_{2}} - (-1)^{\parity{a_{1}}\cdot \parity{a_{2}}} \delta_{s_{2}, r_{1}} E_{r_{2},s_{1}}^{a_{2}a_{1}},
\end{equation}
for all \(r_{1},s_{1}, r_{2}, s_{2} \in [1,n]\) and homogeneous \(a_{1}, a_{2} \in A\).  As we did here, we often only give definitions and formulas on homogeneous elements.  The general case is obtained by extending linearly.

Observe that the general linear Lie algebra $\gl_{n}(\k)$ naturally appears as Lie subsuperalgebra of $\gl_{n}(A)$ consisting of the matrices with entries from $\k$.
 Likewise, if we write \(U(\gl_{n}(A))\) for the universal enveloping superalgebra of \(\gl_n(A)\), then $U(\gl_{n}(\k))$ is a subsuperalgebra of $U(\gl_{n}(A))$.  Both $\gl_{n}(\k )$ and $U(\gl_{n}(\k))$ are purely even.

For short, write $h_{i} = E^{1}_{i,i}$ for $i=1, \dotsc ,n$, and let $\fh$ denote the Cartan subalgebra of $\gl_{n}(\k)$ with basis $h_{1}, \dotsc , h_{n}$.  Let $\varepsilon_{i} \in \fh^{*}= \Hom_{\k}(\fh, \k )$ be the linear functional determined by $\varepsilon_{i}(h_{j}) = \delta_{i,j}$.  Let $\Lambda(n) = \oplus_{i=i}^{n} \Z_{\geq 0}\varepsilon_{i}$.  For $d \geq 0$, let 
\[
\Lambda(n,d) = \left\{\lambda = \sum_{i=1}^{n} \lambda_{i}\varepsilon_{i} \in \Lambda(n) \; \left| \; \sum_{i=1}^{n}\lambda_{i} = d \right.\right\}.
\]
For short, given $\bi = (i_{1}, \dotsc , i_{d}) \in I(n,d)$, write $\bepsilon_{\bi} = \sum_{k=1}^{d} \varepsilon_{i_{k}} \in \Lambda(n,d)$.

\subsection{Examples}\label{SS:Examples} The above setup encompasses a number of cases of interest.  We give a few examples.

If $A= \k$, then $\gl_{n}(\k)$ is the general linear Lie algebra.  If $A = \Cliff:=\k[c]/(c^{2}-1)$ is the Clifford superalgebra on one odd generator, then $\gl_{n}(A) \cong \fq_{n}(\k)$ is the type Q Lie superalgebra. If $\Lambda = \k [x]/(x^{2})$ is the Grassmann algebra on one odd generator, then the Lie superalgebra $\mathfrak{sl}_{2}(\Lambda)$ and its representations appear in the context of annular Khovanov homology \cite{GrigsbyLicataWehrli}.

If $A = \k[t_{1}, \dotsc , t_{k}]$ or $A = \k[t_{1}^{\pm 1}, \dotsc , t_{k}^{\pm 1}]$, then $\gl_{n}(A)$ is the general linear (multi)current Lie algebra or (multi)loop Lie algebra, respectively.  The representations of (multi)loop Lie algebras are closely related to the level zero representations of the general linear affine and toroidal Lie algebras.  Current algebras and their representations are central to the study of Yangians and, more recently, have appeared in the trace decategorifications of categorified quantum groups \cite{MolevBook,BHLW,SVV}.  Generalizing the loop and current Lie algebras, if $A$ is a finitely-generated commutative $\k$-algebra, then $\gl_{n}(A)$ is a so-called map Lie algebra and its representations are an area of active research.  A noncommutative example of this kind is when $A = \k_{q}$, a quantum torus. In this case, $\mathfrak{sl}_{n}(\k_{q})$ appears as the core of the elliptic quasi-simple Lie algebras in type A \cite{BGK}. 

Given a superalgebra $A$ and a superalgebra involution $\iota: A \to A$, one can define an associated superalgebra $\Cliff (A)$, which equals $A \otimes \Cliff$ as a superspace and has product given by $(a_{1} \otimes c^{k_{1}})(a_{2} \otimes  c^{k_{2}}) = (-1)^{\parity{a_{2}} \cdot \parity{k_{1}}} a_{1}\iota^{k_{1}}(a_{2}) \otimes c^{k_{1}+k_{2}}$ for homogeneous $a_{1}, a_{2} \in A$ and $k_{1}, k_{2} \in \{0,1 \}$.  Then one could define $\fq_{n}(A)$ as $\gl_{n}\left( A \otimes \Cliff \right)$.  For example, taking $A = \k$ and $\iota$ to be the identity map recovers the case $\mathfrak{q}_{n}(\k ) \cong \gl_{n}(\Cliff)$ mentioned previously.  Taking $A = \k [t]$ or $A = \k [t^{\pm 1}]$ and $\iota: A \to A$ determined by $t \mapsto -t$, or $A = \k [t^{\pm 1}]$ and $\iota: A \to A$ determined by $t \mapsto t^{-1}$ recovers the twisted Lie superalgebras of type Q studied by Nazarov \cite{Nazarov} and Chen--Guay \cite{ChenGuay}.  The Lie superalgebra $\gl_{n}(A)$ with $A = \k\langle u^{\pm 1},c \rangle /(uu^{-1}=u^{-1}u=1, c^{2}=u, cu=uc)$ with $u$ even and $c$ odd is closely related to the loop algebra for the twisted affine type Q Lie superalgebra $\mathfrak{q}(n)^{(2)}$ studied in, for example, \cite{HoytSerganova,GorelikSerganova}.

\subsection{Representations}  A \emph{$\gl_{n}(A)$-supermodule} is a $\k$-superspace $M = M_{\0}\oplus M_{\1}$ which is a $U(\gl_{n}(A))$-module where the action respects the parity in that $x.m \in M_{\parity{x}+\parity{m}}$ for all homogeneous $x \in U(\gl_{n}(A))$ and $m \in M$.  There is a natural $\Z_{2}$-grading on $\Hom_{\k}(M,N)$ where $f : M\to N$ is declared to have parity $r \in \Z_{2}$ if $f(M_{s}) \subseteq N_{s+r}$ for all $s \in \Z_{2}$.  A homogeneous $\gl_{n}(A)$-supermodule homomorphism is a homogeneous map $f: M \to N$ which satisfies $f(x.m) = (-1)^{\parity{f}\cdot\parity{x}}x.f(m)$ for all homogeneous $x \in U(\gl_{n}(A))$ and $m \in M$.  A general $\gl_{n}(A)$-supermodule homomorphism is then a $\k$-linear map $f:M \to N$ that can be written $f = f_{\0} + f_{\1}$ where $f_{r}$ is a homogeneous $\gl_{n}(A)$-supermodule homomorphism of parity $r\in \Z_{2}$.  In particular, note that we do not assume supermodule homomorphisms preserve parity.

If $M$ is a $\gl_{n}(A)$-supermodule, $m \in M$ is called a \emph{weight vector}  of \emph{weight} $\lambda \in \fh^{*}$ if
\[
h_{i}m = \lambda(h_{i})m
\] for all $i=1, \dotsc , n$.  We write $M_{\lambda}$ for the subspace of all weight vectors of weight $\lambda$, and call $M$ a \emph{weight module} if $M = \oplus_{\lambda \in \fh^{*}} M_{\lambda}$.  The supercommutator formula \eqref{E:supercommutator} implies that 
\[
E_{r,s}^{a}M_{\lambda} \subseteq M_{\lambda + \varepsilon_{r}-\varepsilon_{s}}
\] for all $\lambda \in \fh^{*}$, $1 \leq r,s \leq  n$,  and homogeneous $a \in A$.

Set $V_{n}= A^{\oplus n}$ to be column vectors of height $n$ with entries from $A$.  This is naturally a $\gl_{n}(A)$-supermodule with $\Z_{2}$-grading given by $V_{n, r} = A_{r}^{\oplus n}$ for $r \in \Z_{2}$, and with action given by matrix multiplication.  Since  $U(\gl_{n}(A))$ is a Hopf superalgebra there is a natural action on $V_{n}^{\otimes d}$ for all $d \geq 1$.  We follow the convention that $V_{n}^{\otimes 0} = \k$ is the trivial supermodule.

For $i \in [1,n]$ and homogeneous $a \in A$, write $v_{i}^{a}$ for a column vector with $a$ in the $i$th row, and zeros elsewhere. The parity of $v_{i}^{a}$ is $\parity{a}$.  Using this notation, 
\[
\left\{v_{\bi}^{\bb} := v_{i_{1}}^{b_{1}}\otimes \dotsb \otimes v_{i_{d}}^{b_{d}} \mid \bi \in I(n,d),  \bb \in I(\BasisB,d)  \right\}
\] is a homogeneous $\k$-basis for $V_{n}^{\otimes d}$.  From this it is easy to see that $V_{n}^{\otimes d}$ is a weight module, $v_{i_{1}}^{b_{1}}\otimes \dotsb \otimes v_{i_{d}}^{b_{d}} \in \left(V_{n}^{\otimes d} \right)_{\bepsilon_{\bi}}$, and that the weights of $V_{n}^{\otimes d}$ lie in $\Lambda(n,d)$.

Write $E_{n} = \k^{\oplus n}$ for the natural representation of $\gl_{n}(\k)$; that is, $E_{n}$ consists of column vectors of height $n$ with entries from $\k$ where $\gl_{n}(\k)$ acts by matrix multiplication.  Recall that, for $d \geq 0$, the irreducible polynomial representations of degree $d$ of $\gl_{n}(\k )$ are, by definition, the irreducible representations which appear as a composition factor of $E_{n}^{\otimes d}$.  

For the purposes of this paper, a polynomial representation of degree $d$ for $U(\gl_{n}(\k))$ will mean a $U(\gl_{n}(\k ))$-supermodule which is isomorphic to a (possibly infinite) direct sum of irreducible polynomial representations of degree $d$.  In particular, the category of polynomial representations of degree $d$ for $U(\gl_{n}(\k ))$ is a semisimple category.  If $M$ is a polynomial representation of degree $d$ for $U(\gl_{n}(\k))$, then $M$ is a weight module with 
\[
M = \bigoplus_{\lambda \in \Lambda(n,d)} M_{\lambda}.
\]

\begin{definition}\label{D:polynomial}  For $d \geq 0$, let $\Pol_{d} = \Pol_{d}(\gl_{n}(A))$ be the full subcategory of $U(\gl_{n}(A))$-supermodules $M$ which satisfy the following conditions:
\begin{enumerate}
\item the restriction of $M$ to $U(\gl_{n}(\k ))$ is isomorphic to a polynomial representation of degree $d$;
\item for all $i \in [1,n]$, if $m \in M_{\lambda}$ with $\lambda (h_{i}) =0$, then $E_{i,i}^{a}m =0$ for all homogeneous $a \in A$.
\end{enumerate}  
\end{definition} \noindent   Observe that $V_{n}^{\otimes d}$ and subquotients of $V_{n}^{\otimes d}$ are objects in $\Pol_{d}$.

\begin{remark} Readers familiar with classical Schur--Weyl duality will not be surprised by the first condition in Definition~\ref{D:polynomial}. The second condition already appears in the study of polynomial representations of $\fq_{n}(\k) = \gl_{n}(\Cliff)$ \cite[Definition 1.5]{GJKK}. Furthermore, if $N$ is a $\SdwreathA$-supermodule, then the $U(\gl_{n}(A))$-supermodule $V_{n}^{\otimes d} \otimes_{\SdwreathA} N$ satisfies both conditions.  That is, the image of the functor $F_{n,d}^{A}$ lies in the category $\Pol_{d}(\gl_{n}(A))$.    

The reader may wonder if the second condition is either implied by the first condition or is otherwise an artifact of our argument. The next example shows that the second condition does not follow from being a polynomial representation of $U(\gl_{n}(\k ))$; Example~\ref{Ex:Example2} gives an example of where the supermodule satisfies the first condition but not the second, and the supermodule also fails to be in the image of the functor $F_{n,d}^{A}$. From this point of view, our main result demonstrates that Definition~\ref{D:polynomial} is an intrinsic characterization of the $\gl_{n}(A)$-supermodules which are in the image of the functor $F_{n,d}^{A}$. 
\end{remark}

\begin{example}\label{Ex:Example1}  Let $n  > 1$ and let $A = \k[x]/(x^{2}-1)$ be the purely even group algebra of the cyclic group of order $2$.  If $I_{n}$ denotes the $n\times n$ identity matrix, then $\gl_{n}(A) =  \mathfrak{sl}_{n}(A)  \oplus \k I_{n} \oplus \k xI_{n}$ as Lie algebras, where $I_{n}$ and $xI_{n}$ are both central.  In particular, as an algebra $U(\gl_{n}(A)) \cong U(\mathfrak{sl}_{n}(A) + \k I_{n})\otimes \k[xI_{n}]$.  Let $M=A$ with trivial action by $U(\mathfrak{sl}_{n}(A) + \k I_{n})$ and where $xI_{n}$ acts by multiplication by $x$. Then $M$ is a polynomial representation of $\gl_{n}(\k)$ of degree $0$.  Moreover, every vector in $M$ is of weight zero, but $M$ clearly fails the second condition for modules in $\Pol_{0}$.

As another example, let $W = V_{n} \otimes_{\k} M$, where the action on $M$ is as above, and the action of $U(\gl_{n}(A))$ on $W$ is via the coproduct.  Then $W$ is a polynomial representation of degree $1$ for $U(\gl_{n}(\k ))$.  Consider the vector $w:=v_{1}^{1} \otimes 1$ in $W$.  This vector has weight $\varepsilon_{1}$.  In particular, $\varepsilon_{1}(h_{n})=0$.  But $E_{n,n}^{nx} = xI_{n} - \sum_{i=1}^{n-1}i(E_{i,i}^{x}-E_{i+1,i+1}^{x})$ and so acts on $w$ by $E_{n,n}^{nx}.w =xI_{n}.w - \sum_{i=1}^{n-1}i(E_{i,i}^{x}-E_{i+1,i+1}^{x}).w = v_{1}^{x} \otimes 1 + v_{1}^{1}\otimes x - v_{1}^{x} \otimes 1= v_{1}^{1}\otimes x$. 
\end{example}

\begin{remark}  If $M$ is a polynomial representation of degree $d$ for $U(\gl_{n}(\k))$, then it is integrable and, hence, there is an action by the symmetric group on $n$ letters via the permutation matrices of $\operatorname{GL}_{n}(\k)$.  Using this action one sees that the second condition of Definition~\ref{D:polynomial} holds if and only if it holds for a single $i \in [1,n]$.
\end{remark}

\subsection{Wreath Product Superalgebras}\label{SS:WreathProductSuperalgebras}  Given associative $\k$-superalgebras $A$ and $B$, recall that $A \otimes B$ is a superalgebra with product given by 
\[
(a_{1} \otimes b_{1})(a_{2} \otimes b_{2}) = (-1)^{\parity{b_{1}} \cdot\parity{a_{2}}} (a_{1}a_{2})\otimes (b_{1}b_{2}).
\]  Iterating defines the superalgebra $A_{1} \otimes \dotsb \otimes  A_{d}$ for any $d \geq 1$.  In particular, this defines $A^{\otimes d}$ for any $d \geq 1$.  We follow the convention that $A^{\otimes 0}= \k$.

For $d \geq 0$, there is a right action on $A^{\otimes d}$ by the symmetric group on $d$ letters, $S_{d}$, by signed place permutation:
\[
(a_{1} \otimes a_{2} \otimes \dotsb \otimes a_{i} \otimes a_{i+1} \otimes \dotsb \otimes a_{d}) (i,i+1) = (-1)^{\parity{a_{i}}\cdot\parity{a_{i+1}}} a_{1} \otimes a_{2} \otimes \dotsb \otimes a_{i+1} \otimes a_{i} \otimes \dotsb \otimes a_{d}.
\]

For $d \geq 0$, let $S_{d} \wreath A$ denote the wreath product superalgebra.  It equals 
\[
\k S_{d} \otimes  A^{\otimes d} 
\] as a superspace, where $\k S_{d}$ is viewed as a superalgebra concentrated in parity $\0$.  The product structure is given by the formula 
\[
\left( \sigma \otimes a_{1} \otimes \dotsb \otimes  a_{d} \right)  \cdot \left(\sigma' \otimes a'_{1} \otimes \dotsb \otimes  a'_{d} \right) =  \sigma\sigma' \otimes \left((a_{1} \otimes \dotsb \otimes  a_{d})\sigma' \right) \left(a'_{1} \otimes \dotsb \otimes a'_{d} \right).
\]  In particular, $\k S_{d}$ and $A^{\otimes d}$ are isomorphic to subsuperalgebras of $\SdwreathA$ in obvious ways.


For $\bi \in I(n,d)$ and $\ba \in I(A,d)$, let $\bv_{\bi}^{\ba} = v_{i_{1}}^{a_{1}} \otimes \dotsb \otimes v_{i_{d}}^{a_{d}} \in V_{n}^{\otimes d}$.  Observe that  $V_{n}^{\otimes d}$ is a right $A^{\otimes d}$-supermodule via 
\[
\bv_{\bi}^{\ba} \left( c_{1}\otimes \dotsb \otimes c_{d}\right) = (-1)^{\gamma(\ba, \bc)}\bv_{\bi}^{\ba \cdot \bc},
\] for all $\bi \in I(n,d)$, $\ba \in I(A,d)$, and homogeneous $c_{1} \otimes \dotsb \otimes  c_{d} \in A^{\otimes d}$, where $\ba \cdot \bc := (a_{1}c_{1}, \dotsc , a_{d}c_{d})$ and where $\gamma(\ba, \bc) = \sum_{1 \leq p < q \leq d} \parity{a_{q}}\cdot\parity{c_{p}}$ is obtained by applying the rule of signs.  Under this action $V_{n}^{\otimes d}$ is a free right $A^{\otimes d}$-supermodule with basis $\left\{\bv^{\bone}_{\bi} \mid \bi \in I(n,d) \right\}$.  Also observe that this action can be extended to $S_{d} \wreath A$ by having the elements of $S_{d}$ act by signed place permutation.  It is straightforward to verify the actions of $U(\gl_{n}(A))$ and $\SdwreathA$ on $V_{n}^{\otimes d}$ mutually commute.

\subsection{The Schur--Weyl  Functor}   Since $V_{n}^{\otimes  d}$ is a $(U(\gl_{n}(A)), S_{d}\wreath A)$-bisupermodule,  there is a functor from left $S_{d}\wreath A$-supermodules to left $U(\gl_{n}(A))$-supermodules,
\[
F^{A}=F_{n,d}^{A}: S_{d} \wreath A-\text{supermodules} \to U(\gl_{n}(A))-\text{supermodules},
\] given on supermodules by $N \mapsto V_{n}^{\otimes d} \otimes_{S_{d}\wreath A} N$ and on morphisms by $f \mapsto \Id_{V_{n}^{\otimes d}} \otimes f$.

\begin{remark}\label{R:KeyRemark}  As a special case, note that $F^{\k}= F_{n,d}^{\k}$ is the classical Schur--Weyl functor between $\k S_{d}$-modules and $U(\gl_{n}(\k))$-modules.  We will repeatedly use the well-known fact that $F_{n,d}^{\k}$ gives an equivalence between the category of $\k S_{d}$-supermodules and the category of polynomial supermodules of degree $d$ for $U(\gl_{n}(\k))$ whenever $n \geq d$, where both $\k S_{d}$ and $U(\gl_{n}(\k ))$ are viewed as superalgebras concentrated in parity $\0$  \cite[Chapter 6]{GreenBook}. 
\end{remark}

Since $\k S_{d}$ is a subalgebra of $S_{d} \wreath A$ and $U(\gl_{n}(\k ))$ is a subalgebra of $U(\gl_{n}(A))$, we have corresponding restriction functors and, hence, a square of functors:
\begin{equation}\label{E:functor-square}
\begin{tikzcd}
 S_{d}\wreath A-\textrm{mod} \arrow{rr}{F_{n,d}^{A}} \arrow{d}{res_{\k S_{d}}^{S_{d} \wreath A}} &&U(\gl_{n}(A))-\textrm{mod} \arrow{d}{res_{U(\gl_{n}(\k))}^{U(\gl_{n}(A))}}\\
 \k S_{d}-\textrm{mod}  \arrow{rr}{F_{n,d}^{\k}} &&U(\gl_{n}(\k))-\textrm{mod}
\end{tikzcd}.
\end{equation}
As we next explain, this square commutes up to a canonical isomorphism.

Recall that we write $V_{n}$ for $A^{\oplus n}$ and $E_{n}$ for $\k^{\oplus n}$.  We will write $v_{1}, \dotsc , v_{n}$ for the standard $\k$-basis of $E$, where $v_{i}$ has a $1$ in the ith entry and zeros elsewhere.  Then $\left\{ \bv_{\bi} = v_{i_{1}}\otimes \dotsb \otimes v_{i_{d}} \mid \bi \in I(n,d) \right\}$ is a $\k$-basis for $E_{n}^{\otimes d}$. We identify $E_{n}^{\otimes d}$ as a $\k$-submodule of $V_{n}^{\otimes d}$ via $\bv_{\bi}\mapsto \bv_{\bi}^{\bone}$ for $\bi  \in I(n,d)$.

Because $V_{n}^{\otimes d}$ is a free $A^{\otimes d}$-supermodule with basis $\left\{ \bv_{\bi}^{\bone} \mid \bi \in I(n,d)  \right\}$, it follows that for any $\SdwreathA$-supermodule, $N$, the map 
\begin{equation}\label{E:identifyI}
 E_{n}^{\otimes d}  \otimes_{\k} N \to  V_{n}^{\otimes d}\otimes_{A^{\otimes d}} N 
\end{equation}
determined by $ \bv_{\bi} \otimes_{\k} n \mapsto \bv_{\bi}^{\bone} \otimes_{A^{\otimes d}}  n$ is a superspace isomorphism.  Going forward, we will identify these two superspaces via this map.

Both $E_{n}^{\otimes d}  \otimes_{\k} N$ and $V_{n}^{\otimes d}\otimes_{A^{\otimes d}} N$ are $\k S_{d}$-supermodules where $S_{d}$ acts diagonally on each.  The map given in \eqref{E:identifyI} is $S_{d}$-equivariant.  Hence, it is an isomorphism of $\k S_{d}$-supermodules and restricts to an isomorphism between the $S_{d}$-invariant subspaces.  With this observation we have the middle isomorphism in the following chain:
\begin{equation}\label{E:identifyII}
 E_{n}^{\otimes d}  \otimes_{\k S_{d}} N  \cong\left(  E_{n}^{\otimes d} \otimes_{\k} N \right)^{S_{d}}  \cong \left(  V_{n}^{\otimes d} \otimes_{A^{\otimes d}} N \right)^{S_{d}} \cong  V_{n}^{\otimes d}  \otimes_{\SdwreathA} N.
\end{equation}  The first and last isomorphisms are straightforward.  Observe that all three maps are $U(\gl_{n}(\k))$-supermodule isomorphisms.  Going forward, we will identify $E_{n}^{\otimes d}  \otimes_{\k S_{d}} N$ and $ V_{n}^{\otimes d}  \otimes_{\SdwreathA} N$ as $U(\gl_{n}(\k))$-supermodules via the composite of these maps.

In particular, using this identification one sees that the square \eqref{E:functor-square} commutes for objects: $res_{U(\gl_{n}(\k))}^{U(\gl_{n}(A))}\left( F_{n,d}^{A}(N)\right) = res_{\k S_{d}}^{S_{d} \wreath A}\left( F_{n,d}^{\k}(N)\right)$ for all $\SdwreathA$-supermodules $N$. Likewise, it is straightforward to see that the square commutes on morphisms.

Combining Remark~\ref{R:KeyRemark}, the commutativity of the square \eqref{E:functor-square}, and the fact that the restriction functors are faithful, immediately yields the following.

\begin{proposition}\label{P:SWisFaithful} The functor $F_{n,d}^{A}$ is faithful whenever $n \geq d$.
\end{proposition}

\section{The Schur--Weyl Equivalence}

The goal of this section is to prove that the functor 
\[
F^{A}_{n,d}: S_{d}\wreath A-\textrm{mod} \to \Pol_{d}(\gl_{n}(A))
\] is an equivalence of categories when $n\geq d+1$.  As remarked after Definition~\ref{D:polynomial}, the image of $F^{A}_{n,d}$ lies in $\Pol_{d}(\gl_{n}(A))$.  Proposition~\ref{P:SWisFaithful} shows $F^{A}_{n,d}$ is faithful.  Therefore, all that remains is to show it is full and essentially surjective.

\subsection{Key Injectivity Result}\label{SS:Key-injectivity-result}  The following will be used repeatedly in the sequel.
\begin{proposition}\label{T:injective-theorem}
Let $n,d \geq 1$ with $n \geq d$.  Fix $\bi = (i_1,\dotsc , i_d) \in \tI(n,d)$.  Set $\bv^{\bone}_{\bi}=v^{1}_{i_1} \otimes \dotsb  \otimes v^{1}_{i_d}$.  Then the map $N \rightarrow V_{n}^{\otimes d}  \otimes_{\SdwreathA} N$ defined by $n \mapsto \bv^{\bone}_{\bi }  \otimes n$ is injective for any $\SdwreathA$-module $N$.
\end{proposition}

\begin{proof} Let $S$ be the $\k$-span of the set
\[
\left\{ v_{j_{1}}^{c_{1}}\otimes \dotsb \otimes v_{j_{d}}^{c_{d}} \mid \left\{j_{1}, \dotsc , j_{d} \right\} = \left\{i_{1}, \dotsc , i_{d} \right\}, c_{1}, \dotsc , c_{d} \in A \right\}, 
\] and let $S^{\perp}$ be the $\k$-span of the set
\[
\left\{ v_{j_{1}}^{c_{1}}\otimes \dotsb \otimes  v_{j_{d}}^{c_{d}} \mid \left\{j_{1}, \dotsc , j_{d} \right\} \neq \left\{i_{1}, \dotsc , i_{d} \right\}, c_{1}, \dotsc , c_{d} \in A \right\}.
\]
Then,
\[
V_{n}^{\otimes d} = S \oplus S^{\perp}
\] is a decomposition into right $\SdwreathA$-supermodules, and  
\begin{equation}\label{E:Decomp}
V_{n}^{\otimes d} \otimes_{\SdwreathA} N = \left( S \otimes_{\SdwreathA} N\right) \oplus \left(  S^{\perp} \otimes_{\SdwreathA} N\right)
\end{equation}
as superspaces.

Since the entries of $\bi \in \tI(n,d)$ are all distinct, $S$ is a free right $\SdwreathA$-module of rank one generated by $\bv^{\bone}_{\bi}$. As a consequence, $S \otimes_{\SdwreathA} N =\bv^{\bone}_{\bi }  \otimes N$ is isomorphic to $N$ as a superspace with isomorphism given by $n \mapsto \bv^{\bone}_{\bi } \otimes n$. Coupled with the decomposition given in \eqref{E:Decomp}, this shows the map in the claim is injective.
\end{proof}

\subsection{Fullness}\label{}

\begin{theorem}\label{T:FullFunctor} The functor $F_{n,d}^{A}$ is full whenever $n \geq d$.
\end{theorem}

\begin{proof} Let $N_{1}$ and $N_{2}$ be $\SdwreathA$-supermodules and let $g \in \Hom_{U(\gl_{n}(A))}\left( V_{n}^{\otimes d} \otimes_{\SdwreathA} N_{1}, V_{n}^{\otimes d} \otimes_{\SdwreathA} N_{2}\right)$ be a homogeneous map.  Applying the restriction functor and the identification given by \eqref{E:identifyII}, we may view $g$ as a morphism in $\Hom_{U(\gl_{n}(\k))}\left( E_{n}^{\otimes d} \otimes_{\k S_{d}} N_{1}, E_{n}^{\otimes d} \otimes_{\k S_{d}} N_{2}\right)$.  Since $F_{n,d}^{\k}$ is a full functor, we may select a homogeneous $\k S_{d}$-supermodule homomorphism $f: N_{1} \to N_{2}$ of the same parity as $g$ such that 
\[
\Id_{E_{n}^{\otimes d}} \otimes f = g.
\]  That is, $\Id_{V_{n}^{\otimes d}} \otimes f = g$ on the subspace of $V_{n}^{\otimes d} \otimes_{\SdwreathA} N$ spanned by $\left\{v_{\bi}^{\bone} \otimes n \mid \bi \in I(n,d), n \in N \right\}$.  We claim that $f$ is actually a $\SdwreathA$-supermodule homomorphism.  That is, $f$ is an allowable input into the functor $F_{n,d}^{A}$. This, along with the identification given in \eqref{E:identifyII}, demonstrates that $F_{n,d}^{A}(f) = g$, as desired.

For short, given $k \in [1, d]$ and homogeneous $c \in A$, write $c_{k}$ for $1^{\otimes (k-1)} \otimes c \otimes 1^{\otimes (d-k-1)}  \in A^{\otimes d}$.  Since these elements along with $\k S_{d}$ generate $\SdwreathA$, and since $f$ is already known to be $\k S_{d}$-equivariant, it will suffice to show for all homogeneous $n \in N_{1}$ and all homogeneous $c \in A$  that 
\[
f \left( c_{k}n \right) = (-1)^{\parity{f}\cdot \parity{c_{k}}} c_{k}  f(n).
\] 

Since $n \geq d$, we may set $\bv = v^{1}_{1} \otimes \dotsb \otimes v^{1}_{d}$.  Consider
\begin{align*}
\bv \otimes f \left( c_{k}n \right) &= \left(\Id_{V_{n}^{\otimes d}} \otimes f \right) \left(\bv \otimes c_{k}n \right) \\
&= g\left(\bv \otimes c_{k}n \right) \\
&= g\left(\bv c_{k} \otimes n \right) \\
&= g\left( (v_{1}^1 \otimes \dotsb \otimes v_{k-1}^{1}\otimes  v_{k}^{c} \otimes v_{k+1}^{1} \otimes \dotsb \otimes v_{d}^{1}) \otimes  n \right) \\
&= g \left(\left( E_{k,k}^{c}( v^{1}_{1} \otimes \dotsb \otimes v_{k-1}^{1}\otimes  v_{k}^{1} \otimes v_{k+1}^{1} \otimes \dotsb \otimes v_{d}^{1}) \right) \otimes  n\right)\\
&= (-1)^{\parity{g}\cdot \parity{E_{k,k}^{c}}} E_{k,k}^{c} g \left( (v_{1}^1 \otimes \dotsb \otimes v_{k-1}^{1}\otimes  v_{k}^{1} \otimes v_{k+1}^{1} \otimes \dotsb \otimes v_{d}^{1}) \otimes  n\right) \\
&=(-1)^{\parity{g}\cdot \parity{E_{k,k}^{c}}} E_{k,k}^{c} \left( \Id_{V_{n}^{\otimes d}} \otimes f\right)  \left((v_{1}^1 \otimes \dotsb \otimes v_{k-1}^{1}\otimes  v_{k}^{1} \otimes v_{k+1}^{1} \otimes \dotsb \otimes v_{d}^{1}) \otimes  n\right)\\
&=(-1)^{\parity{g}\cdot \parity{E_{k,k}^{c}}} E_{k,k}^{c}\left( (v_{1}^1 \otimes \dotsb \otimes v_{k-1}^{1}\otimes  v_{k}^{1} \otimes v_{k+1}^{1} \otimes \dotsb \otimes v_{d}^{1}) \otimes  f(n)\right)\\
&=(-1)^{\parity{g}\cdot \parity{E_{k,k}^{c}}} (v_{1}^1 \otimes \dotsb \otimes v_{k-1}^{1}\otimes  v_{k}^{c} \otimes v_{k+1}^{1} \otimes \dotsb \otimes v_{d}^{1}) \otimes  f(n)\\
&=(-1)^{\parity{g}\parity{E_{k,k}^{c}}} (v_{1}^1 \otimes \dotsb \otimes v_{k-1}^{1}\otimes  v_{k}^{1} \otimes v_{k+1}^{1} \otimes \dotsb \otimes v_{d}^{1})c_{k} \otimes  f(n) \\
&=(-1)^{\parity{g}\parity{E_{k,k}^{c}}} (v_{1}^1 \otimes \dotsb \otimes v_{k-1}^{1}\otimes  v_{k}^{1} \otimes v_{k+1}^{1} \otimes \dotsb \otimes v_{d}^{1}) \otimes  c_{k}f(n)\\
&= (-1)^{\parity{g}\parity{E_{k,k}^{c}}} \bv \otimes  c_{k}f(n).
\end{align*}  By Proposition~\ref{T:injective-theorem}, 
\[
f \left( c_{k}n \right) = (-1)^{\parity{g} \cdot \parity{E_{k,k}^{c}}}c_{k} f(n).
\]  Finally, note that $\parity{f}=\parity{g}$ and $\parity{E_{k,k}^{c}} = \parity{c}_{k}$, and thus the desired equality holds.
\end{proof}

\subsection{Defining a $\SdwreathA$-module structure}\label{SS:lifting-the-wreath-product-structure}  

Let $n \geq d+1$.  Assume $N$ is a $\k S_{d}$-supermodule for which $E_{n}^{\otimes d} \otimes_{\k S_{d}}N$ has a $U(\gl_{n}(A))$-supermodule structure that makes it an object of $\Pol_{d}(\gl_{n}(A))$ and, furthermore, that action restricts to the standard action of $U(\gl_{n}(\k ))$ on $E_{n}^{\otimes d} \otimes_{\k S_{d}}N$.  Throughout this section,  we assume we are in this situation.

The goal of this section is to show that with this setup it is possible to extend the $\k S_{d}$-action on $N$ to an action by $\SdwreathA$ on $N$.  The general strategy is to use the assumed action of $\gl_{n}(A)$ on $E_{n}^{\otimes d} \otimes_{\k S_{d}}N$ to define the action of $A^{\otimes d}$ on $N$.  Very roughly, given $a \in A$ and $n \in N$, we would like to define the action of $a_{k}:=1^{\otimes (k-1)} \otimes a \otimes 1^{ \otimes d-k}$ on $n$ using the action of $E_{j_{k}, i_{k}}^{a}$ on $v_{\bi} \otimes n$ for $(\bj, \bi)_{k} \in \tilde{I}^{2}(n,d)$ in combination with Proposition~\ref{T:injective-theorem}.  Let us explain how to implement this strategy.

For $\bi \in I(n,d)$, write $\bv_{\bi} = v_{i_{1}}\otimes\dotsb \otimes v_{i_{d}} \in E_{n}^{\otimes d}$.  Given $ (\bj , \bi )_{k} \in \Itwo$ and a homogeneous $a \in A$, define a $\k$-linear map, 
\[
\alpha_{(\bj, \bi)_{k}}^{a}: N \to N,
\] using the action of $E_{j_{k}, i_{k}}^{a}$ on $E_{n}^{\otimes d} \otimes_{\k S_{d}} N$, as follows.   By weight considerations, for any homogeneous $n \in N$ we have
\[
E_{j_{k},i_{k}}^{a}(\bv_{\bi} \otimes n)  \in \left(E_{n}^{\otimes d} \otimes_{\k S_{d}} N \right)_{\bepsilon_{\bj}} = \left(E_{n}^{\otimes d}\right)_{\bepsilon_{\bj}} \otimes_{\k S_{d}} N.
\] However, because $\bj$ has distinct entries, $\left(E_{n}^{\otimes d}\right)_{\bepsilon_{\bj}}$  is a free right $\k S_{d}$-supermodule of rank one with basis vector $\bv_{\bj}$.  Therefore, 
\[
E_{j_{k},i_{k}}^{a}(\bv_{\bi} \otimes n) = \sum_{\sigma \in S_{d}} \left(  \bv_{\bj}\sigma \otimes n_{\sigma}\right) =  \bv_{\bj} \otimes \left(\sum_{\sigma \in S_{d}} \sigma  n_{\sigma}\right) = \bv_{\bj} \otimes  \alpha_{(\bj , \bi)_{k}}^{a} (n),
\]  for some $\alpha_{(\bj , \bi)_{k}}^{a} (n) \in N$.  Since $\bj$ has distinct entries, Proposition~\ref{T:injective-theorem} implies $\alpha_{(\bj , \bi)_{k}}^{a}(n)$ is unique and, hence, the rule $n \mapsto \alpha_{(\bj , \bi)_{k}}^{a}(n)$ defines a function, as claimed. The fact that it is a $\k$-linear map of parity $\parity{a}$ is straightforward to check.

The following result records some of the basic properties of the maps $\alpha^{a}_{(\bj , \bi)_{k}}:N \to N$.  Note that here and going forward, we frequently make use of the observation that if $M$ is a supermodule for $U(\gl_{n}(A))$ which lies in $\Pol_{d}(\gl_{n}(A))$, and if $m \in M_{\lambda}$ and $\lambda (h_{j}) = 0$, then $E_{i,j}^{a}m=0$: either because $i=j$ and the second condition of Definition~\ref{D:polynomial}  applies, or because $i \neq j$ and because $\lambda +\varepsilon_{i} - \varepsilon_{j}$ is not in $\Lambda(n,d)$ and, hence, is not an allowable weight of $M$.


\begin{lemma}  The following statements are true:
\begin{enumerate}
\item If $(\bj, \bi)_{k}, (\bj', \bi)_{k} \in \Itwo$, then $\alpha^{a}_{(\bj , \bi)_{k}}=\alpha^{a}_{(\bj' , \bi)_{k}}$ for all homogeneous $a \in A$.
\item If $(\bj, \bi)_{k}, (\bj, \bi' )_{k} \in \Itwo$, then $\alpha^{a}_{(\bj , \bi)_{k}}=\alpha^{a}_{(\bj , \bi')_{k}}$ for all homogeneous $a \in A$.
\item If $(\bj, \bi)_{k} \in \Itwo$, then for any $p \in [1,n]\setminus \{i_{1}, \dotsc , i_{d}, j_{1}, \dotsc, j_{d} \}$ and $\ell \in [1,d]$ with $\ell \neq k$, if we define $\bi '=(i_{1}, \dotsc , i_{\ell-1}, p, i_{\ell +1}, \dotsc , i_{d})$ and $\bj'=(j_{1}, \dotsc , j_{\ell-1}, p, j_{\ell +1}, \dotsc , j_{d}) $, then $(\bj', \bi' )_{k} \in \Itwo$ and $\alpha^{a}_{(\bj , \bi)_{k}}=\alpha^{a}_{(\bj' , \bi')_{k}}$ for all homogeneous $a \in A$.
\end{enumerate}
\end{lemma}
\begin{proof}
To prove (1), we first consider the special case of the pair $(\bj, \bi)_{k}$ and $(\bi, \bi)_{k}$.  The claim in (1) is trivial in this case if $\bj = \bi$, so we assume $\bj  \neq \bi$ and, hence, $j_{k} \neq i_{k}$. Use the  supercommutator formula \eqref{E:supercommutator}  to see that
\[
E_{i_{k}, j_{k}}^{1}E_{j_{k}, i_{k}}^{a} = E_{j_{k},i_{k}}^{a}E_{i_{k},j_{k}}^{1} + E_{i_{k},i_{k}}^{a} - E^{a}_{j_{k}, j_{k}}.
\]
Letting both sides act on on $\bv^{}_{\bi} \otimes n$ yields 
\[
E_{i_{k}, j_{k}}^{1}E_{j_{k}, i_{k}}^{a}(\bv^{}_{\bi} \otimes n) = E_{i_{k}, j_{k}}^{1}(\bv^{}_{\bj} \otimes \alpha^{a}_{(\bj , \bi)_{k}}(n))= \bv^{}_{\bi} \otimes \alpha^{a}_{(\bj , \bi)_{k}}(n)
\] and 
\begin{align*}
\left(E_{j_{k},i_{k}}^{a}E_{i_{k},j_{k}}^{1} + E_{i_{k},i_{k}}^{a} - E^{a}_{j_{k}, j_{k}} \right)(\bv^{}_{\bi} \otimes n) &= E_{j_{k},i_{k}}^{a}E_{i_{k},j_{k}}^{1}(\bv^{}_{\bi} \otimes n) + E_{i_{k},i_{k}}^{a}(\bv^{}_{\bi} \otimes n) - E^{a}_{j_{k}, j_{k}}(\bv^{}_{\bi} \otimes n) \\
&= E_{j_{k},i_{k}}^{a}E_{i_{k},j_{k}}^{1}(\bv^{}_{\bi} \otimes n) + \bv^{}_{\bi} \otimes \alpha^{a}_{(\bi, \bi)_{k}}(n) - E^{a}_{j_{k}, j_{k}}(\bv^{}_{\bi} \otimes n).
\end{align*}
However, $\bepsilon_{\bi} + \varepsilon_{i_{k}}-\varepsilon_{j_{k}} \not \in \Lambda(n,d)$ and so the first term of the final sum equals zero; likewise, the observation that $h_{j_{k}}\bv_{\bi}=0$ along with the second property of supermodules in $\Pol_{d}(\gl_{n}(A))$ implies the last term is also equal to zero.  Therefore, $\bv^{}_{\bi} \otimes \alpha^{a}_{(\bj , \bi)_{k}}(n) = \bv^{}_{\bi} \otimes \alpha^{a}_{(\bi, \bi)_{k}}(n)$.  It follows from Proposition~\ref{T:injective-theorem} that $\alpha^{a}_{(\bj , \bi)_{k}}(n) = \alpha^{a}_{(\bi, \bi)_{k}}(n)$, as desired.

Applying the special case to the pair $(\bj , \bi )_{k}$ and $(\bi,\bi)_{k}$, and the pair $(\bj' , \bi )_{k}$ and $(\bi,\bi)_{k}$, yields $\alpha^{a}_{(\bj , \bi)_{k}} = \alpha^{a}_{(\bi,\bi)_{k}}$ and $ \alpha^{a}_{(\bj', \bi)}= \alpha^{a}_{(\bi,\bi)_{k}}$, respectively.  Combining the two gives  $\alpha^{a}_{(\bj , \bi)_{k}}  = \alpha^{a}_{(\bj', \bi)}$, as desired.

Likewise, to prove (2), it suffices to consider the case of $(\bj, \bi)_{k}$ and $(\bj, \bj )_{k}$ where $\bj \neq \bi$.  In particular, $j_{k} \neq i_{k}$.  Use the supercommutator formula \eqref{E:supercommutator} to see that 
\[
E^{a}_{j_{k}, j_{k}}E^{1}_{j_{k}, i_{k}} = E^{1}_{j_{k}, i_{k}}E^{a}_{j_{k}, j_{k}} + E^{a}_{j_{k}, i_{k}}
\]  Applying each side to $ \bv_{\bi}^{} \otimes n$ yields
\[
E^{a}_{j_{k}, j_{k}}E^{1}_{j_{k}, i_{k}}( \bv_{\bi}^{} \otimes n) = E^{a}_{j_{k}, j_{k}}\left( \bv_{\bj}^{} \otimes n \right) =  \bv_{\bj}^{} \otimes \alpha_{(\bj , \bj)_{k}}^{a} (n)
\]
and 
\begin{align*}
\left( E^{1}_{j_{k}, i_{k}}E^{a}_{j_{k}, j_{k}} + E^{a}_{j_{k}, i_{k}}\right)( \bv_{\bi} \otimes n) &= \left( E^{1}_{j_{k}, i_{k}}E^{a}_{j_{k}, j_{k}}\right)( \bv_{\bi} \otimes n) +  \bv_{\bj} \otimes  \alpha_{(\bj, \bi)_{k}}^{a}(n).
\end{align*}  However, the first term in the last sum is again zero by the second property of supermodules in $\Pol_{d}(\gl_{n}(A))$.  Proposition~\ref{T:injective-theorem} then implies $\alpha_{(\bj , \bj)_{k}}^{a} (n) = \alpha_{(\bj, \bi)_{k}}^{a}(n)$, as claimed.

To prove (3), we first note that, by the choice of $p$, the entries of $\bi'$ and $\bj'$ are pairwise distinct and the two have the same values in every position except possibly at $i'_{k}$ and $j'_{k}$.  Therefore $(\bi', \bj')_{k} \in \Itwo$, as claimed.  Noting that $j'_{k}=j_{k}$ and $i'_{k}=i_{k}$, and, hence, that $i'_{k}\neq p$ and $j'_{k}\neq i_{\ell}$, the supercommutator formula \eqref{E:supercommutator} implies that 
\[
E^{a}_{j'_{k},i'_{k}}E_{p,i_{\ell}}^{1} =E_{p,i_{\ell}}^{1}E^{a}_{j'_{k},i'_{k}}.
\] Applying each side to $ \bv_{\bi}^{} \otimes n$ yields
\[
\left(E^{a}_{j'_{k},i'_{k}}E_{p,i_{\ell}}^{1} \right)( \bv_{\bi}^{} \otimes n) = E^{a}_{j'_{k},i'_{k}}( \bv_{\bi'}^{} \otimes n)= \bv_{\bj'}^{} \otimes \alpha_{(\bj', \bi')_{k}}^{a}(n) 
\] and 
\[
E_{p,i_{\ell}}^{1}E^{a}_{j'_{k},i'_{k}}( \bv_{\bi}^{} \otimes n) = E_{p,i_{\ell}}^{1}E^{a}_{j_{k},i_{k}}( \bv_{\bi}^{} \otimes n)=E_{p,j_{\ell}}^{1}\left( \bv_{\bj}^{} \otimes \alpha_{\bj, \bi}^{a} (n) \right) =  \bv_{\bj'}^{} \otimes \alpha_{\bj, \bi}^{a} (n).
\]  Proposition~\ref{T:injective-theorem} implies $\alpha_{(\bj' , \bi')_{k}}^{a} (n) = \alpha_{(\bj, \bi)_{k}}^{a}(n)$, as claimed.
\end{proof}

\begin{proposition} \label{L:independence-of-alpha} Let $n \geq d+1$.  For $(\bj, \bi)_{k}, (\bj', \bi')_{k} \in \Itwo$ and homogeneous $a \in A$, 
\[
\alpha_{(\bj', \bi')_{k}}^{a} = \alpha_{(\bj, \bi)_{k}}^{a}.
\]  That is, the map $\alpha_{(\bj, \bi)_{k}}^{a}$ may depend on $k$ and $a$, but is independent of the choice of $(\bj, \bi)_{k} \in \Itwo$.
\end{proposition}
\begin{proof} 

For a fixed homogeneous $a \in A$, define an equivalence relation on $\Itwo$ by declaring $(\bj, \bi)_{k} \sim (\bj', \bi')_{\ell}$ if and only if $k=\ell$ and $\alpha_{(\bj', \bi')_{k}}^{a} = \alpha_{(\bj, \bi)_{k}}^{a}$. It suffices to show that $(\bj, \bi)_{k} \in \Itwo$ is equivalent to $((1,2,\dotsc  ,d),(1,2,\dotsc,d))_{k}$.  By applying (1) from the previous lemma if necessary, we may assume without loss of generality that $(\bj , \bi )_{k} = (\bi, \bi)_{k}$.

To prove $(\bi, \bi)_{k}$ is equivalent to $((1,2,\dotsc  ,d),(1,2,\dotsc,d))_{k}$, let 
\[
D \left((\bi, \bi)_{k}\right) := \# \left\{t \in [1,d] \mid i_{t} \neq t \right\} 
\]  be the function which counts the number of entries of $(\bi, \bi)_{k}$ that are different than the corresponding entries of $((1,2,\dotsc ,d),(1,2,\dotsc,d))_{k}$.  We argue by induction on $D \left((\bi, \bi)_{k}\right)$, with the base case of $D \left((\bi, \bi)_{k}\right)=0$ being trivial.

Assume $D \left((\bi, \bi)_{k}\right) > 0$.  Then $i_{\ell}\neq \ell$ for some $\ell \in [1,d]$.  We first consider the case when $\ell$ does not appear as an entry of $(\bi, \bi)_{k}$.  If $\ell \neq k$, then we may use (3) from the previous lemma to simultaneously change $i_{\ell}$ of both tuples to $\ell$ and thereby obtain a pair $(\bi', \bi')_{k} \in  \Itwo$ which is equivalent to $(\bi, \bi)_{k}$ and with $D((\bi', \bi')_{k}) = D((\bi, \bi)_{k})-1$.  If $\ell = k$, let $\bi'= (i_{1}, \dotsc , i_{k-1}, \ell, i_{k+1}, \dotsc , i_{d})$.  Using (1) shows that $\alpha^{a}_{(\bi, \bi)_{k}}=\alpha^{a}_{(\bi', \bi)_{k}}$  and (2) shows $\alpha^{a}_{(\bi', \bi)_{k}}=\alpha^{a}_{(\bi', \bi')_{k}}$.  Once again, we have obtained a pair $(\bi', \bi')_{k} \in  \Itwo$ which is equivalent to $(\bi, \bi)_{k}$ and with $D((\bi', \bi')_{k}) = D((\bi, \bi)_{k})-1$.

On the other hand, if $\ell$ appears as an entry of $(\bi, \bi)_{k}$, then first choose a $q \in [1,n]$ which does not appear as an entry of $(\bi, \bi)_{k}$.  Such a $q$ exists because  $(\bi, \bi)_{k}$ has $d$ distinct entries and, by assumption, $n \geq  d+1$.  Arguing as in the previous paragraph, we may replace $\ell$ with $q$ wherever it appears in the pair $(\bi, \bi)_{k}$.  In this way we obtain a pair $(\bi'', \bi'')_{k} \in  \Itwo$ that is equivalent to $(\bi, \bi)_{k}$, satisfies $D((\bi'' , \bi'')_{k}) \leq D((\bi , \bi )_{k})$, and where $\ell$ does not appear as an entry of $(\bi'' , \bi'')_{k}$. We are now in the situation of the previous paragraph.  That argument yields a pair $(\bi ' , \bi ')_{k} \in \Itwo$ that is equivalent to $(\bi'', \bi'')_{k}$ and, hence, is equivalent to $(\bi , \bi)_{k}$.  Furthermore, $D((\bi', \bi')_{k}) = D((\bi'', \bi'')_{k})-1 < D((\bi, \bi)_{k})$.

By induction, $(\bi, \bi)_{k}$ is equivalent to $((1,2,\dotsc  ,d),(1,2,\dotsc,d))_{k}$ and this proves the claim.
\end{proof}

In light of the previous result, we sometimes write $\alpha^{a}_{k}$ for $\alpha^{a}_{(\bj , \bi)_{k}}$, where the $(\bj , \bi )_{k} \in \Itwo$ is left implicit.

We continue to assume that $N$ is a $\k S_{d}$-supermodule such that $E_{n}^{\otimes d} \otimes_{\k S_{d}} N$ has a $U(\gl_{n}(A))$-supermodule structure that makes it a supermodule in $\Pol_{d}(\gl_{n}(A))$, and such that the  action of the subalgebra $U(\gl_{n}(\k))$ on $E_{n}^{\otimes d} \otimes_{\k S_{d}} N$ is the obvious one.

\begin{proposition} \label{L:alpha-is-an-algebra-map} Let $n \geq d+1$.   The map 
\[
\alpha : A^{\otimes d} \to \End_{\k}(N)
\] defined by 
\[
\alpha(c_{1}\otimes \dotsb \otimes c_{d}) = \alpha_{1}^{c_{1}} \circ \dotsb \circ \alpha_{d}^{c_{d}}
\] is a superalgebra homomorphism.  Moreover, the action on $N$ by $A^{\otimes d}$ afforded by this map along with the given action of $\k S_{d}$ on $N$ defines a $\SdwreathA$-module structure on $N$.
\end{proposition}
\begin{proof} We first observe that $\alpha_{k}^{1}:=\alpha_{k}^{1_{A}}$ is the identity map for all $k \in [1,d]$.  Namely, fix $(\bj, \bi)_{k}\in \Itwo$ and then 
\[
 \bv_{\bj}^{} \otimes \alpha_{k}^{1}(n)   = E_{j_{k}, i_{k}}^{1} ( \bv_{\bi} \otimes n) =  \bv_{\bj}^{} \otimes n.
\]  From Proposition~\ref{T:injective-theorem}, $\alpha_{k}^{1}(n) = n$, as claimed.

It follows that 
\[
\alpha (1^{\otimes (k-1)} \otimes  a \otimes 1^{\otimes (d-k)}) = \alpha_{1}^{1} \circ \dotsb \circ \alpha_{k-1}^{1} \circ \alpha_{k}^{a} \circ \alpha_{k+1}^{1} \circ\dotsb \circ \alpha_{d}^{1} = \alpha_{k}^{a}
\]
for all $k \in [1,d]$ and all homogeneous $a \in A$.  Since the various $1^{\otimes (k-1)} \otimes  a \otimes 1^{\otimes (d-k)}$ generate $A^{\otimes d}$, it suffices to verify:
\begin{enumerate}
\item for all homogeneous $a, a' \in A$, $\alpha_{k}^{a+a'} = \alpha_{k}^{a}+\alpha_{k}^{a'}$ ;
\item for $c \in \k$ and homogeneous $a \in A$, $\alpha_{k}^{ca} = c\alpha_{k}^{a}$ ;
\item for all homogeneous $a, a' \in A$, if $k \neq k'$, then
\[
\alpha_{k}^{a} \circ \alpha_{k'}^{a'}  = (-1)^{\parity{a} \cdot \parity{a'}} \alpha_{k'}^{a'} \circ \alpha_{k}^{a} \; ;\]
\item for all homogeneous $a,a' \in A$, $\alpha_{k}^{a} \circ \alpha_{k}^{a'} = \alpha_{k}^{aa'}$ .
\end{enumerate}

The first two follow by picking a $(\bj , \bi )_{k} \in \Itwo$, computing $\alpha_{(\bj , \bi )_{k}}^{a+a'}$ and $\alpha_{(\bj , \bi )_{k}}^{ca}$ from their definitions, and applying the fact that $E_{j_{k}, i_{k}}^{a+a'} = E_{j_{k}, i_{k}}^{a}+E_{j_{k}, i_{k}}^{a'}$ and $E_{i_{k}, j_{k}}^{ca} = cE_{i_{k}, j_{k}}^{a}$.

The third follows by letting $\bi \in \tI (n,d)$, computing $\alpha_{(\bi , \bi )_{k}}^{a}$ and $\alpha^{a'}_{(\bi , \bi )_{k'}}$ from their definitions, and applying the fact that $E_{i_{k}, i_{k}}^{a}E_{i_{k'}, i_{k'}}^{a'} = (-1)^{\parity{a}\cdot\parity{a'}}E_{i_{k'}, i_{k'}}^{a'}E_{i_{k}, i_{k}}^{a}$.

We next explain the fourth claim.  Pick $(\bj , \bi )_{k}, (\bi , \bi )_{k} \in \Itwo$ with $i_{k}$ and $j_{k}$ distinct. Two such pairs exist because $n \geq  d+1$.   We wish to show $\alpha_{(\bj , \bi)_{k}}^{a} \circ \alpha_{(\bi, \bi)_{k}}^{a'} = \alpha_{(\bj , \bi )_{k}}^{aa'}$. Using the supercommutator formula \eqref{E:supercommutator}, we have
\[
E_{j_{k}, i_{k}}^{a}E_{i_{k}, i_{k}}^{a'} = (-1)^{\parity{a}\cdot \parity{a'}} E_{i_{k}, i_{k}}^{a'}E_{j_{k}, i_{k}}^{a} + E_{j_{k}, i_{k}}^{aa'}.
\]  Applying both sides to $\bv^{}_{\bi}  \otimes n$ yields, 
\[
\left(E_{j_{k}, i_{k}}^{a}E_{i_{k}, i_{k}}^{a'} \right) ( \bv^{}_{\bi} \otimes n) = E_{j_{k},i_{k}}^{a}(\bv^{}_{\bi} \otimes \alpha_{(\bi , \bi )_{k}}^{a'}(n)) = \bv^{}_{\bj} \otimes \alpha_{(\bj , \bi )_{k}}^{a}(\alpha_{(\bi , \bi )_{k}}^{a'}(n)),
\] and, 
\begin{align*}
\left((-1)^{\parity{a}\cdot \parity{a'}} E_{i_{k}, i_{k}}^{a'}E_{j_{k}, i_{k}}^{a} + E_{j_{k}, i_{k}}^{aa'} \right) (\bv^{ }_{\bi} \otimes n) &= \left((-1)^{\parity{a}\cdot \parity{a'}}E_{i_{k}, i_{k}}^{a'}E_{j_{k}, i_{k}}^{a}\right) (\bv^{ }_{\bi} \otimes n) + E_{j_{k}, i_{k}}^{aa'} (\bv^{ }_{\bi} \otimes n) \\
& = (-1)^{\parity{a}\cdot \parity{a'}} E_{i_{k}, i_{k}}^{a'}   (\bv^{ }_{\bj} \otimes \alpha^{a}_{(\bj , \bi)_{k}} (n)) +  \bv^{ }_{\bj} \otimes \alpha_{(\bj, \bi )_{k}}^{aa'}(n).
\end{align*}
However, by the second property of supermodules in $\Pol_{d}(\gl_{n}(A))$ the first term in the last sum equals zero.  Proposition~\ref{T:injective-theorem} shows that $\alpha_{(\bj , \bi )_{k}}^{a}(\alpha_{(\bi , \bi )_{k}}^{a'}(n))= \alpha_{(\bj, \bi )_{k}}^{aa'}(n)$, as desired.

Since $\alpha$ is a superalgebra map, we can view $N$ as a $A^{\otimes d}$-supermodule via this map. Since $N$ is already assumed to have a $\k S_{d}$-supermodule structure, it remains to show that these two actions are compatible and so define a $\SdwreathA$-supermodule structure on $N$.  To do this it suffices to verify that if $\sigma_{p} = (p, p+1) \in S_{d}$ is a simple transposition, then 
\[
\alpha_{p+1}^{a}   =  \sigma_{p} \circ \alpha_{p}^{a} \circ \sigma_{p}
\] as endomorphisms of $N$.

Fix $(\bj , \bi)_{p} \in \Itwo$. Set $\bj' = (j_{1}, \dotsc j_{p-1}, j_{p+1}, j_{p}, j_{p+2}, \dotsc , j_{d})$ and $\bi' = (i_{1}, \dotsc i_{p-1}, i_{p+1}, i_{p}, i_{p+2}, \dotsc , i_{d})$.  Then $(\bj', \bi')_{p+1} \in \Itwo$.  Under the place permutation action of $S_{d}$ on $E_{n}^{\otimes d}$, we have  $\bv^{}_{\bj'}\sigma_{p} = \bv^{}_{\bj}$ and  $\bv^{}_{\bi}\sigma_{p} = \bv^{}_{\bi'}$.  We then calculate:
\begin{align*}
\bv^{}_{\bj'} \otimes \sigma_{p}\left(\alpha_{(\bj,\bi)_{p}}^{a}(\sigma_{p}(n)) \right) &= \bv^{}_{\bj'}\sigma_{p} \otimes  \alpha_{(\bj,\bi)_{p}}^{a}(\sigma_{p}(n)) \\
&= \bv^{}_{\bj} \otimes  \alpha_{(\bj,\bi)_{p}}^{a}(\sigma_{p}(n)) \\
&= E^{a}_{j_{p}, i_{p}}(\bv^{}_{\bi} \otimes \sigma_{p}(n)) \\
&= E^{a}_{j_{p},i_{p}}(\bv^{}_{\bi}\sigma_{p}\otimes n)\\
&= E^{a}_{j_{p},i_{p}}(\bv^{}_{\bi'}\otimes n)\\
&= E^{a}_{j'_{p+1}, i'_{p+1}}(\bv^{}_{\bi'}\otimes n)\\
&=\bv^{}_{\bj'} \otimes \alpha^{a}_{(\bj' , \bi')_{p+1}}(n).
\end{align*} Proposition~\ref{T:injective-theorem} implies $\sigma_{p}(\alpha_{p}^{a}(\sigma_{p}(n))) =\alpha^{a}_{p+1}(n)$, as desired.
\end{proof}

\subsection{Recovering the action of $\gl_{n}(A)$}\label{SS:Recovering-the-glnA-action} 
We continue to assume that $N$ is a $\k S_{d}$-supermodule such that $E_{n}^{\otimes d} \otimes_{\k S_{d}} N$ has a $U(\gl_{n}(A))$-supermodule structure that makes it a supermodule in $\Pol_{d}(\gl_{n}(A))$, and such that the  action of the subalgebra $U(\gl_{n}(\k))$ on $E_{n}^{\otimes d} \otimes_{\k S_{d}} N$ is the obvious one.  In the previous section, the action of $U(\gl_{n}(A))$ on $E_{n}^{\otimes d} \otimes_{\k S_{d}} N$ was used on certain special elements to define maps  $\alpha^{a}_{k}:N \to N$ for homogeneous $a \in A$ and $k\in [1,d]$.  We now show that these maps completely determine the action of $U(\gl_{n}(A))$ on $E_{n}^{\otimes d} \otimes_{\k S_{d}} N$.

\begin{proposition}\label{P:compatibile-actions} Let $n \geq d+1$ and let $E_{n}^{\otimes d} \otimes_{\k S_{d}} N$ be as above and let $n \in N$. For homogeneous $a \in A$, $p,q \in [1,n]$, and $\bi \in I(n,d)$, we have 
\begin{align}\label{E:action}
E^{a}_{p,q} (v^{}_{\bi}\otimes n) &=  \sum_{k=1}^{d} v^{}_{i_{1}} \otimes \dotsb \otimes v^{}_{i_{k-1}}\otimes  E^{1}_{p,q} v^{}_{i_{k}} \otimes v^{}_{k+1} \otimes \dotsb \otimes v^{}_{i_{d}} \otimes \alpha^{a}_{k}(n) \\ \notag
&= \sum_{k=1}^{d} \delta_{q, i_{k}} v^{}_{i_{1}} \otimes \dotsb \otimes v^{}_{i_{k-1}}\otimes   v^{}_{p} \otimes v^{}_{k+1} \otimes \dotsb \otimes v^{}_{i_{d}} \otimes \alpha^{a}_{k}(n) 
\end{align}

\end{proposition}

\begin{proof}  That the two sums are the same is immediate from the action of $\gl_{n}(\k )$ on $E_{n}$.  We claim they equal the left-hand side of the expression.  We argue by induction on $D(\bi): = \# \{ 1 \leq  r < s \leq d \mid i_{r}= i_{s}\}$, the number of pairs of entries in $\bi$ that have the same value.

The base case is when all entries of $\bi$ are distinct and $D(\bi ) =0$. If $q$ does not appear as an entry of $\bi$, then the left side of \eqref{E:action} is zero because $\bepsilon_{\bi}+ \varepsilon_{p}-\varepsilon_{q}$ is not an element of $\Lambda(n,d)$, and the right side is also zero because $\delta_{q,i_{k}}=0$ for all $k$.  If $q$  appears as an entry of $\bi$, then it appears only once --- say at position $\ell \in [1,d]$.  Let $\bj =(i_{1}, \dotsc , i_{\ell-1}, p, i_{\ell+1}, \dotsc , d)$.  In this case, the left hand side equals $\bv^{}_{\bj} \otimes \alpha^{a}_{(\bj , \bi)_{\ell}}(n)=\bv^{}_{\bj} \otimes \alpha^{a}_{\ell}(n)$ and, by direct calculation the right hand side also equals $\bv^{ }_{\bj } \otimes  \alpha^{a}_{\ell}(n)$.  In both cases \eqref{E:action} is verified when $D(\bi) = 0$ for all $E^{a}_{p,q}$.

We now assume $D(\bi ) > 0$ and that \eqref{E:action} holds for all $E^{a}_{p,q}$ and all $v^{}_{\bj} \otimes n$ whenever $D(\bj ) < D(\bi )$.  Since $D(\bi) > 0$, there exists $1 \leq k_{1} < k_{2} \leq d$ such that $i_{k_{1}}=i_{k_{2}}$.  Let $\ell = k_{1}$.  Since $n \geq d+1$, we may choose $x \in [1,n]$ so that $x$ does not appear as an entry of $\bi$.  Set $\bj = (i_{1}, \dotsc , i_{\ell-1}, x, i_{\ell+1}, \dotsc , i_{d})$.  Observe that $D(\bj ) < D(\bi )$.  From the supercommutator formula \eqref{E:supercommutator}, we have 
\begin{equation}\label{E:action-commutator}
E^{a}_{p,q}E^{1}_{i_{\ell}, x} =E^{1}_{i_{\ell}, x} E^{a}_{p,q} + \delta_{q, i_{\ell}} E^{a}_{p,x} - \delta_{p,x}E^{a}_{i_{\ell}, q}.
\end{equation}
Applying the left side to $\bv^{}_{\bj} \otimes n$ gives:
\[
E^{a}_{p,q}E^{1}_{i_{\ell}, x}(\bv^{}_{\bj} \otimes n) = E^{a}_{p,q}(\bv^{}_{\bi} \otimes n).
\]  This is the left-hand side of \eqref{E:action}.

On the other hand, apply the right-hand side of \eqref{E:action-commutator} to $\bv^{}_{\bj} \otimes n$ and consider each term separately.  By induction, the first term is:
\begin{align*}
E^{1}_{i_{\ell}, x} E^{a}_{p,q} (\bv^{}_{\bj}\otimes n) &=E^{1}_{i_{\ell}, x} \left(  \sum_{k=1}^{d} \delta_{q, j_{k}} v_{j_{1}} \otimes \dotsb \otimes  v_{j_{k-1}}\otimes v_{p} \otimes v_{j_{k+1}} \otimes  \dotsb \otimes v_{j_{d}}  \otimes \alpha_{k}^{a}(n)\right) \\
&=E^{1}_{i_{\ell}, x} \left(  \sum_{k\neq \ell} \delta_{q, i_{k}} v_{i_{1}} \otimes \dotsb \otimes v_{i_{k-1}} \otimes v_{p} \otimes v_{i_{k+1}} \otimes \dotsb \otimes v_{i_{\ell-1}}\otimes  v_{x} \otimes v_{i_{\ell+1}}\otimes  \dotsb \otimes v_{i_{d}}  \otimes \alpha_{k}^{a}(n) \right) \\
 &\hspace{0.5in} + E^{1}_{i_{\ell}, x} \left( \delta_{q,x} v_{i_{1}}\otimes \dotsb \otimes v_{i_{\ell-1}} \otimes v_{p} \otimes v_{i_{\ell+1}} \otimes \dotsb \otimes  v_{i_{d}} \otimes \alpha_{\ell}^{a}(n)\right) \\
&=  \sum_{k\neq \ell} \delta_{q, i_{k}}\delta_{p,x} v_{i_{1}} \otimes \dotsb \otimes v_{i_{k-1}}\otimes  v_{i_{\ell}} \otimes v_{i_{k+1}} \otimes \dotsb \otimes v_{i_{\ell-1}}\otimes  v_{x} \otimes v_{i_{\ell+1}}\otimes  \dotsb \otimes v_{j_{d}}  \otimes \alpha_{k}^{a}(n)  \\
 &\hspace{0.5in}  + \sum_{k\neq \ell} \delta_{q, i_{k}} v_{i_{1}} \otimes \dotsb\otimes v_{i_{k-1}} \otimes v_{p} \otimes v_{i_{k+1}}\otimes  \dotsb \otimes v_{i_{\ell-1}}\otimes  v_{i_{\ell}} \otimes v_{i_{\ell+1}}\otimes  \dotsb \otimes v_{j_{d}}  \otimes \alpha_{k}^{a}(n) \\
 &\hspace{0.5in}   +  \delta_{p,x} \delta_{q,x} v_{i_{1}}\otimes \dotsb v_{i_{\ell-1}} \otimes v_{i_{\ell}} \otimes v_{i_{\ell+1}} \otimes \dotsb v_{i_{d}} \otimes  \alpha_{\ell}^{a}(n).
\end{align*}  By induction, the second term is:
\begin{align*}
\delta_{q, i_{\ell}} E^{a}_{p,x}(\bv^{}_{\bj}\otimes n) &=   \sum_{k=1}^{d} \delta_{q, i_{\ell}} \delta_{x,j_{k}}v_{j_{1}} \otimes \dotsb \otimes v_{j_{k-1}} \otimes v_{p} \otimes v_{j_{k+1}} \otimes  \dotsb \otimes v_{j_{d}}  \otimes \alpha_{k}^{a}(n) \\
&=  \delta_{q, i_{\ell}} v_{i_{1}} \otimes \dotsb \otimes v_{i_{\ell-1}} \otimes v_{p} \otimes v_{i_{\ell+1}} \otimes \dotsb \otimes v_{d} \otimes \alpha^{a}_{\ell}(n).
\end{align*} By induction, the third term is:
\begin{align*}
\delta_{p,x}E^{a}_{i_{\ell},q}(\bv^{}_{\bj}\otimes n) &=  \sum_{k=1}^{d} \delta_{p, x}\delta_{q,j_{k}} v_{j_{1}} \otimes \dotsb  \otimes v_{j_{k-1}}\otimes v_{i_{\ell}} \otimes v_{j_{k+1}}\otimes  \dotsb \otimes v_{j_{d}}  \otimes \alpha_{k}^{a}(n) \\
&=\sum_{k\neq \ell}\delta_{p, x}\delta_{q, i_{k}} v_{i_{1}} \otimes \dotsb \otimes v_{i_{k-1}}\otimes  v_{i_{\ell}}\otimes v_{i_{k+1}} \otimes \dotsb \otimes v_{i_{\ell-1}} \otimes v_{x} \otimes v_{i_{\ell}+1}\otimes \dotsb \otimes  v_{j_{d}}  \otimes \alpha_{k}^{a}(n) \\
& \hspace{0.5in} + \delta_{p, x}\delta_{q,x} v_{i_{1}}\otimes \dotsb \otimes v_{i_{\ell-1}}\otimes v_{i_{\ell}}\otimes v_{i_{\ell+1}} \otimes \dotsb \otimes v_{i_{d}} \otimes  \alpha^{a}_{\ell}(n).
\end{align*}
Combining these three expressions and simplifying yields the right side of \eqref{E:action}.
\end{proof}

\subsection{The Functor $F^{A}_{n,d}$ is Essentially Surjective}\label{SS:Essential-surjectivity}

We can now prove that $F^{A}_{n,d}$ is essentially surjective whenever $n \geq d+1$.

\begin{theorem}\label{T:essentially-surjective}  Let $n \geq d+1$.  The functor 
\[
F^{A}_{n,d}: S_{d}\wreath A-\textrm{mod} \to \Pol_{d}(\gl_{n}(A))
\] is essentially surjective.
\end{theorem}

\begin{proof} Let $M$ be an object of $\Pol_{d}(\gl_{n}(A))$.  From classical Schur--Weyl duality combined with Lemma~\ref{L:alpha-is-an-algebra-map} there is a $\SdwreathA$-supermodule $N$ so that $E_{n}^{\otimes d} \otimes_{\k S_{d}} N$ is isomorphic to $M$ as a $U(\gl_{n}(\k ))$-supermodule and, under that isomorphism, the action of $U(\gl_{n}(A))$ on $E_{n}^{\otimes d} \otimes_{\k S_{d}} N$ is given by Proposition~\ref{P:compatibile-actions}.

It suffices to verify for this $\SdwreathA$-supermodule, $N$, that the $U(\gl_{n}(\k ))$-supermodule isomorphism 
\[
V_{n}^{\otimes d} \otimes_{\SdwreathA} N \cong E_{n}^{\otimes d} \otimes_{\k S_{d}} N
\]
given by \eqref{E:identifyII} defines an isomorphism of $U(\gl_{n}(A))$-supermodules.  That is, the action of $U(\gl_{n}(A))$ on these two superspaces coincides under this isomorphism.

Recall that $V_{n}^{\otimes d} \otimes_{\SdwreathA} N$ is spanned by elements of the form $\bv^{\bone }_{\bi} \otimes n$ for $\bi \in I(n,d)$ and homogeneous $n\in N$.  On these vectors the action of $E_{p,q}^{a}$ for $p,q \in [1,d]$ and homogeneous $a \in A$ is given by: 
\begin{align*}
E^{a}_{p,q} (\bv^{\bone}_{\bi}\otimes n) &= \sum_{k=1}^{d} v^{1}_{i_{1}} \otimes \dotsb \otimes v^{1}_{i_{k-1}} \otimes E^{a}_{p,q}v^{1}_{i_{k}} \otimes v^{1}_{i_{k+1}} \otimes \dotsb \otimes v^{1}_{i_{d}} \otimes n \\
&= \sum_{k=1}^{d} \delta_{q, i_{k}}v^{1}_{i_{1}} \otimes \dotsb \otimes v^{1}_{i_{k-1}} \otimes v^{a}_{p} \otimes v^{1}_{i_{k+1}} \otimes \dotsb \otimes v^{1}_{i_{d}} \otimes n \\
&= \sum_{k=1}^{d} \delta_{q, i_{k}}(v^{1}_{i_{1}} \otimes \dotsb \otimes v^{1}_{i_{k-1}} \otimes v^{1}_{p} \otimes v^{1}_{i_{k+1}} \otimes \dotsb \otimes v^{1}_{i_{d}})a_{k} \otimes n \\
&= \sum_{k=1}^{d} \delta_{q, i_{k}}(v^{1}_{i_{1}} \otimes \dotsb \otimes v^{1}_{i_{k-1}} \otimes v^{1}_{p} \otimes v^{1}_{i_{k+1}} \otimes \dotsb \otimes v^{1}_{i_{d}}) \otimes a_{k}n\\
&= \sum_{k=1}^{d} \delta_{q, i_{k}}(v^{1}_{i_{1}} \otimes \dotsb \otimes v^{1}_{i_{k-1}} \otimes v^{1}_{p} \otimes v^{1}_{i_{k+1}} \otimes \dotsb \otimes v^{1}_{i_{d}}) \otimes \alpha^{a}_{k}(n),
\end{align*} where $a_{k} := 1^{\otimes (k-1)} \otimes a \otimes 1^{\otimes (d-k)}$.   Comparing this calculation with Proposition~\ref{P:compatibile-actions} confirms the actions coincide. 
\end{proof}

\subsection{Main Theorem}\label{SS:MainTheorem}

Combining Theorems~\ref{P:SWisFaithful}, \ref{T:FullFunctor}, and \ref{T:essentially-surjective} gives our main result. 
\begin{theorem}\label{T:main-theorem}  Let $n \geq d+1$.  The functor  
\[
F^{A}_{n,d}: \SdwreathA-\textrm{mod} \to \Pol_{d}(\gl_{n}(A))
\] is an equivalence of categories.
\end{theorem}

\begin{remark}\label{R:finite-dimensional-case}
 If $n \geq d+1$ and $N$ is a $\SdwreathA$-supermodule, then $F^{A}_{n,d}(N) \cong E_{n}^{\otimes d} \otimes_{\k S_{d}} N$ is finite-dimensional if and only if $N$ is finite-dimensional. Thus, the restriction of $F^{A}_{n,d}$ defines an equivalence between the full subcategories of finite-dimensional supermodules.
\end{remark}

\subsection{Some Consequences}  We give a few example applications of the main theorem.  For simplicity's sake, throughout this section we assume $A$ is finite-dimensional and that $\k$ is a splitting field for $A$ (e.g., $\k$ is algebraically closed).

\begin{corollary}\label{C:semisimple}  For all $n \geq d+1$,  the category $\Pol_{d}(\gl_{n}(A))$ is semisimple if and only if $A$ is semisimple.
\end{corollary}

\begin{proof} Write $J(A)$ for the Jacobson radical of the superalgebra $A$. By \cite[Lemma 2.6]{BK} this is a homogeneous ideal and coincides with the Jacobson radical of $A$ as an algebra. The arguments of \cite[Section 3]{CT}  show that
\[
(\SdwreathA )/J(\SdwreathA ) \cong S_{d} \wr \left(A/J(A) \right).
\]  From this we see the $\SdwreathA$ is semisimple if and only if $A$ is semisimple.  The result follows.
\end{proof}

By \cite[Section 6]{CT}, if $A$ is quasi-hereditary algebra, then so is $\SdwreathA$ and, hence, the category of $\SdwreathA$-modules is a highest weight category.  Thus, we have the following result.

\begin{corollary}\label{C:quasi-hereditary}  If $A$ is a quasi-hereditary, then for all $n \geq d+1$ the category $\Pol_{d}(\gl_{n}(A))$ is a highest weight category.
\end{corollary}

For $k \geq 0$, let $\mathbf{P}(k)$ be the set of all partitions of $k$.  For short, if $\mu = (\mu_{1}, \mu_{2}, \dotsc  )$ is a partition of some nonnegative integer, write $|\mu| = \sum_{r} \mu_{r}$ for that integer.  We say a partition $\mu$ is \emph{strict} if for all $r \geq 1$, $\mu_{r} = \mu_{r+1}$ implies $\mu_{r}=0$.

Say $A$ is a finite-dimensional $\k$-superalgebra and that $S = \left\{S_{1}, \dotsc , S_{t} \right\}$ is a complete, irredundant set of simple $A$-supermodules up to isomorphism.  Recall that a simple $A$-supermodule may either be of type M or of type Q. It is known that the simple $\SdwreathA$-supermodules are indexed by the set 
\[
\Pi^{d}(A):= \left\{\pi: S \to \cup_{k \geq 0} \mathbf{P}_{k} \; \left| \; \sum_{r=1}^{t} |\pi (S_{r})|=d \text{ and $\pi(S_{r})$ is strict whenever $S_{r}$ is of type Q} \right.  \right\}.
\] For example, see \cite{MacD} for the construction of the simples when $A$ is an algebra, and see \cite{RossoSavage} or \cite{DKMZ} for the construction of the simples when $A$ is a superalgebra. Write $S^{\pi}$ for the simple $\SdwreathA$-supermodule labelled by $\pi \in \Pi^{d}(A)$.  We note that the type of $S^{\pi}$ can be determined from the function $\pi$ and that $F^{A}_{n,d}$ preserves the type of simple supermodules. Since we do not use this information, we omit the details.

\begin{corollary}\label{C:Classifying-simples}  The set 
\[
\left\{L_{n}^{A}(\pi) : = F^{A}_{n,d}(S^{\pi}) \mid \pi \in \Pi^{d}(A) \right\}
\] is a complete irredundant set of simple supermodules in $\Pol_{d}(\gl_{n}(A))$ when $n \geq d+1$.  
\end{corollary}

\subsection{Sharpness of the Bound}\label{SS:Sharpness}  Given that the equivalence of categories holds when $n\geq d$ in the classical (i.e., $A= \k$) case, we could ask if the bound $n \geq d+1$ here is an artifact of the proof.  The following example shows this is not the case.  A similar observation about the bound was made for $A=\k [t^{\pm 1}]$ in \cite{Flicker}.



\begin{example}\label{Ex:Example2}  Let $d=n$, and let $A= \k[x]/(x^{2}-1)$ be the group algebra of the cyclic group of order two.  We use the notation established in Example~\ref{Ex:Example1}.    


Let $L= \k u$ be a one-dimensional vector space viewed as a $U(\gl_{n}(A))$-module by declaring that  $x.u=0$ for $x \in \mathfrak{sl}_{n}(A)$, $I_{n}.u=nu$, and $xI_{n}. u= u$.
Observe that $L$ is a polynomial representation of $U(\gl_{n}(\k ))$ of degree $d=n$.  The spanning vector has weight $\bepsilon_{\bi}=\varepsilon_{1}+\dotsb +\varepsilon_{n}$ and, since $\bepsilon_{\bi}(h_{i}) = 1$ for $i=1, \dotsc , n$, this module vacuously satisfies the second condition of Definition~\ref{D:polynomial}. Thus, $L$ is in $\Pol_{d}(\gl_{n}(A))$.  In fact, it is not difficult to see that $L$ is in the image of $F^{A}_{n,d}$.

Let $T$ be a two-dimensional $\k$-vector space viewed as a $\mathfrak{sl}_{n}(A)$-module with the trivial action.  Because $U(\gl_{n}(A)) \cong U(\mathfrak{sl}_{n}(A)) \otimes \k[I_{n}] \otimes \k [xI_{n}]$, we may extend the action of $U(\mathfrak{sl}_{n}(A))$ to $U(\gl_{n}(A))$  by
\[
I_{n} \mapsto \left( \begin{matrix}
n   & 0     \\
0   & n  
\end{matrix}\right),
\hspace{0.5in}
xI_{n} \mapsto \left( \begin{matrix}
1   & 1     \\
0   &   1
\end{matrix}\right).
\]

Observe that $T$ restricted to $U(\gl_{n}(\k))$ is a polynomial representation of degree $d=n$.  It again vacuously satisfies the second condition of Definition~\ref{D:polynomial}. Therefore, $T$ is in $\Pol_{d}(\gl_{d}(A))$ and is a non-split extension of $L$ with itself.  On the other hand, $\SdwreathA$ is a semisimple algebra.  Therefore, $\Pol_{d}(\gl_{d}(A))$ and $\SdwreathA$-mod cannot be equivalent in this case.
\end{example}

This example also illustrates another difference with the classical situation.  Let 
\[
S^{A}(n,d) = \End_{\SdwreathA}\left(V_{n}^{\otimes d} \right)
\] be the generalized Schur algebra.  As in the classical case, there is a Schur functor defined via an idempotent, 
\[
{S}^{A}_{n,d}: S^{A}(n,d)-\textrm{mod} \to \SdwreathA -\textrm{mod}.
\]
This is an equivalence of categories whenever $n \geq d$.  See \cite[Section~5]{EvseevKleshchev} for details.

By \cite[Lemma 11.1.1]{DKMZ}, there is a surjective algebra map 
\[
\zeta_{n,d}: U(\gl_{n}(A)) \to S^{A}(n,d).
\]  By inflation through this map one can view the category of $S^{A}(n,d)$-supermodules as a full subcategory of $U(\gl_{n}(A))$-supermodules.  In the classical situation, this subcategory is $\Pol_{d}(\gl_{n}(\k))$.  The previous example shows that it can be a proper subcategory of $\Pol_{d}(\gl_{n}(A))$ when $n=d$.

\makeatletter
\renewcommand*{\@biblabel}[1]{\hfill#1.}
\makeatother

\bibliographystyle{alpha}
\bibliography{schur-weyl-duality-and-wreath-products}

\end{document}